\documentclass[final,3p,times]{elsarticle}

\usepackage{amssymb}
\usepackage{amsmath}
\usepackage{amsfonts}
\usepackage{amsthm}
\usepackage[english]{babel}
\usepackage{color}
\usepackage{lmodern}
\usepackage{pstricks-add}
\usepackage{dsfont}
\usepackage{hyperref}
\hypersetup{hidelinks}
\usepackage{mathrsfs} 
\usepackage[cal=boondoxo]{mathalfa}

\theoremstyle{plain}% Theorem-like structures provided by amsthm.sty
\newtheorem{theorem}{Theorem}[section]
\newtheorem{lemma}[theorem]{Lemma}

\newtheorem{proposition}[theorem]{Proposition}

\theoremstyle{definition}

\theoremstyle{remark}
\newtheorem{remark}{Remark}

    \DeclareMathOperator\Span{span}
    \DeclareMathOperator\sign{sign}
    
    \DeclareMathOperator\hor{H}
    \DeclareMathOperator\HS{HS}
    \DeclareMathOperator\trace{Tr}
\begin{document}

\begin{frontmatter}



\title{Decay estimates for the linear damped wave equation on the Heisenberg group}


%\author[Pisa]{Alessandro Palmieri}
%
%\address[Pisa]{ Department of Mathematics, University of Pisa, Largo B. Pontecorvo 5, 56127 Pisa, Italy}

\author{Alessandro Palmieri}

\address{Department of Mathematics, University of Pisa, Largo B. Pontecorvo 5, 56127 Pisa, Italy}

\begin{abstract}
This paper is devoted to the derivation of $L^2$ - $L^2$ decay estimates for the solution  of the homogeneous linear damped wave equation on the Heisenberg group $\mathbf{H}_n$, for its time derivative and for its horizontal gradient. Moreover, we consider the improvement of these estimates when further $L^1(\mathbf{H}_n)$ regularity is required for the Cauchy data. %This means that if $u$ is a solution to $(\partial_t^2 -\Delta_{\hor} +\partial_t)u=0$, $u(0,\cdot)=u_0$, $\partial_t u(0,\cdot)=u_1$, where $\Delta_{\hor}$ is the sub-Laplacian on $\mathbf{H}_n$, then we will estimate the $L^2(\mathbf{H}_n)$ - norms of $u(t,\cdot)$ and $\partial_tu(t,\cdot), \nabla_{\hor}u(t,\cdot)$ by assuming $u_0\in  L^1(\mathbf{H}_n)\cap H^1(\mathbf{H}_n)$ and $u_1\in L^1(\mathbf{H}_n) \cap L^2(\mathbf{H}_n) $.
 Our approach will rely strongly on the group Fourier transform of $\mathbf{H}_n$ and on the properties of the Hermite functions that form a maximal orthonormal system for $L^2(\mathbb{R}^n)$ of eigenfunctions of the harmonic oscillator. 
\end{abstract}




\end{frontmatter}


\section{Introduction}

In this work, we derive decay estimates for the solution of the damped wave equation on the Heisenberg group, for its time derivative and for its horizontal gradient. In other words, we estimate the $L^2(\mathbf{H}_n)$ - norm  of the solution $u$ to the linear Cauchy problem
\begin{align}\label{damped wave equation Heisenberg}
\begin{cases}
\partial_t^2 u(t,\eta)-\Delta_{\hor} u(t,\eta)+\partial_t u(t,\eta)= 0, & \eta\in \mathbf{H}_n, \ t>0, \\
u(0,\eta)= u_0(\eta), & \eta\in \mathbf{H}_n, \\
\partial_t u(0,\eta)= u_1(\eta), & \eta\in \mathbf{H}_n,
\end{cases}
\end{align} and its first order derivatives $\partial_t u$ and $\nabla_{\hor}u$.

In the last decades, several papers have been devoted to the study of linear PDE in not-Euclidean structures. Let us recall some results from the liturature in the case of Lie group for two equations, the heat equation and the wave equation, which are close, in some sense, to the equation that we will study in this paper, the damped wave equation.  The properties of the solutions to the heat equation on the Heisenberg group or on more general H-type groups (and of the corresponding \emph{heat kernels}) are investigated in \cite{Gav77,BGG00,GL17,BC18}. In the case of compact Lie group we refer to \cite{Feg78,Feg91} or to \cite{Gre13} for the special case $\mathbb{S}^{2n+1}$. For the description of the heat kernel in a nilpotent Lie group we refer to the pioneering work of Folland \cite{Fol75} and to \cite{VSC92}. On the other hand, for the wave equation on the Heisenberg group or on H-type groups we have to mention \cite{Nac82,MS99,BGX00,GHK02,MS15} while in the framework of compact group we cite \cite{BP11,GR15}. In the case of graded groups, we refer the recent works \cite{Tar18,RT17,RY18}, where the well-posedness of the wave equation is studied in the case of time-dependent speed of propagation $a=a(t)$ in different function spaces, depending on the assumption on $a$.

The main idea in what follows is to generalize some well-know facts in the phase space analysis (\emph{WKB analysis}) for the Euclidean case (cf. \cite{ER18} for an overview on this topic) to the case of the Heisenberg group.

The \emph{Heisenberg group} is the Lie group $\mathbf{H}_n=\mathbb{R}^{2n+1}$ equipped with the multiplication rule
\begin{equation*}
(x,y,\tau)\circ (x',y',\tau')= \left(x+x',y+y',\tau+\tau'+\tfrac{1}{2} (x\cdot y'-x'\cdot y)\right),
\end{equation*} where $\cdot$ denotes the standard scalar product in $\mathbb{R}^n$. %The identity element for $\mathbf{H}_n$ is $0$ and $\eta^{-1}=-\eta$ for any $\eta\in \mathbf{H}_n$.
 A system of left-invariant vector fields that span the Lie algebra $\mathfrak{h}_n$ is given by
\begin{align*}
 \ X_j  \doteq  \partial_{x_j}-\frac{y_j}{2}  \,\partial_\tau,   \ Y_j  \doteq  \partial_{y_j}+\frac{x_j}{2}  \, \partial_\tau ,  \ T\doteq \partial_\tau,
\end{align*} where $1\leq j\leq n$. This system satisfies the commutation relations
\begin{align*}
[X_j,Y_k]=\delta_{jk} \, T \quad \mbox{for}  \ 1\leq j,k \leq n.
\end{align*} Therefore, $\mathfrak{h}_n$ admits the stratification $\mathfrak{h}_n= V_1\oplus V_2$, where $V_1\doteq \Span\{X_j,Y_j\}_{1\leq j\leq n}$ and $V_2\doteq \Span\{T\}$. This means that $\mathbf{H}_n$ is a 2 step stratified Lie group, whose homogeneous dimension is $\mathcal{Q}=2n+2$.
The \emph{sub-Laplacian} on $\mathbf{H}_n$ is defined as
\begin{align}
\Delta_{\hor} & \doteq \sum_{j=1}^n (X^2_j+Y_j^2)
= \sum_{j=1}^n \big(\partial_{x_j}^2+\partial_{y_j}^2\big)+\frac{1}{4} \sum_{j=1}^n \big(x_j^2+y_j^2\big)\partial_{\tau}^2   +\sum_{j=1}^n\Big(x_j \, \partial_{y_j\tau}^2-y_j \,\partial_{ x_j \tau}^2\Big).\label{def sub laplacian}
\end{align} 
For a function $v:\mathbf{H}_n\to \mathbb{R}$, the \emph{horizontal gradient} of $v$ is 
\begin{align*}
\nabla_{\hor} v \doteq (X_1 v, \cdots, X_n v, Y_1 v, \cdots, Y_n v) \equiv \sum_{j=1}^n ((X_j v)X_j +(Y_j v)Y_j),
\end{align*} where each fiber of the horizontal subbundle can be endowed with a scalar product  in such a way that the  horizontal vector fields $X_1,\cdots,X_n,Y_1,\cdots,Y_n$ are orthonormal in each point $\eta\in \mathbf{H}_n$. For a function $v\in L^2(\mathbf{H}_n)$ we say that $X_j v, Y_j v\in L^1_{\mathrm{loc}}(\mathbf{H}_n)$ exist \emph{in the sense of distributions}, if the integral relations
\begin{align*}
\int_{\mathbf{H_n}} \big(X_j v \big)(\eta) \, \phi(\eta) \, \mathrm{d}\eta = \int_{\mathbf{H_n}}  v (\eta) \, \big(X_j^* \phi\big)(\eta) \, \mathrm{d}\eta \quad \mbox{and} \quad \int_{\mathbf{H_n}} \big(Y_j v \big)(\eta) \, \phi(\eta) \, \mathrm{d}\eta = \int_{\mathbf{H_n}}  v (\eta) \, \big(Y_j^* \phi\big)(\eta) \, \mathrm{d}\eta
\end{align*} are fulfilled for any $\phi \in \mathcal{C}^\infty_0(\mathbf{H}_n)$, where $X_j^*=-X_j$ and $Y_j^*=-Y_j$ denote the formal adjoint operators of $X_j$ and $Y_j$, respectively. Therefore, in our framework, the \emph{Sobolev space} $H^1(\mathbf{H}_n)$ is the set of all functions $v\in L^2(\mathbf{H}_n)$ such that $X_j v, Y_j v$ exist in the sense of distributions and $X_j v, Y_j v \in  L^2(\mathbf{H}_n) $ for any $j=1,\cdots,n$.

%Recently, several partial differential equations have been study in non-Euclidean structures such as Carnot groups (stratified Lie groups) or, more generally, homogeneous nilpotent Lie groups. 

%\textcolor{red}{Lavori di Ruzhansky e Georgiev-Palmieri equazione del calore, pi\`{u} alcuni dei lavori citati in \cite{RT18} su onde, Schrondinger e altre equazioni dispersive} \\
%
%\textcolor{red}{Breve storia della WKB analysis e di come si sia mossa negli ultimi anni verso equazioni con coefficienti dipendenti dal tempo o verso sistemi accoppiati nella parte lineare}

In \cite{RT18}, some ideas of the phase space analysis have been applied to the study of semilinear damped wave equation with  mass on graded Lie groups $\mathbb{G}$ for a positive Rockland operator $\mathcal{R}$, namely, the semilinear Cauchy problem 
\begin{align}\label{damped wave equation wih mass Rockland }
\begin{cases}
\partial_t^2 u(t,x)+\mathcal{R} u(t,x)+\partial_t u(t,x)+u(t,x)= |u|^p, & x\in \mathbb{G}, \ t>0, \\
u(0,x)= u_0(x), & x\in \mathbb{G}, \\
\partial_t u(0,x)= u_1(x), & x\in \mathbb{G}.
\end{cases}
\end{align}
 The main tool for this purpose is the group Fourier transform on $\mathbb{G}$. In order to deal with this family of bounded operators $\widehat{u}(\pi)$ associated to  irreducible unitary representations $\pi \in \widehat{\mathbb{G}}$  on some Hilbert spaces $H_{\pi}$, the authors of \cite{RT18} consider for each of these bounded operators its projection on a suitable basis of eigenfunctions for $H_\pi$ that allows to  \textquotedblleft diagonalize\textquotedblright $\,$ the operator given by the action of the infinitesimal representation $\mathrm{d}\pi$  on $\mathcal{R}$. As particular case, the Heisenberg group is considered with $\mathcal{R}=-\Delta_{\hor}$. Due to the presence of a mass term in the linear part of the damped wave equation, an exponential decay rate is obtained for the $L^2(\mathbb{G})$-norms of the solution of the corresponding homogeneous linear problem and its derivatives. Therefore, thanks to this extremely fast decay rate they can work with data just on $L^2(\mathbb{G})$ basis (and regularity related to the order of the derivative, of course). As a consequence, they need only the fact that $\mathrm{d}\pi (\mathcal{R})$ has discrete positive spectrum for any nontrivial unitary representation $\pi$ of $\mathbb{G}$ and they do not have to specifying neither the growth rate of the eigenvalues, nor the action of $\mathrm{d}\pi$ on some of the generators of the Lie algebra of $\mathbb{G}$ to estimates the $L^2(\mathbb{G})$-norm of derivatives with respect to $x$, nor an explicit description of Plancherel measure.
 
 Nevertheless, if we do remove the mass term in \eqref{damped wave equation wih mass Rockland }, the situation changes drastically, since we cannot get exponential decay. In order to find still some decay rate, in this case of polynomial kind depending on the order and on the type of the derivative, we shall require further additional $L^1$ regularity for the Cauchy data. However, we will restrict our considerations to the case of the Heisenberg group with the sub-Laplacian,  where the irreducible unitary representations are given, up to intertwining operators, by Schr\"odinger representations whose infinitesimal representations act on the generators of the first layer of $\mathfrak{h}_n$ in a known way, the Plancherel measure can be expressed through the parameter that parameterizes Schr\"odinger representations  and the operator that has to be diagonalized is simply the harmonic oscillator (cf. Section \ref{Section group fourier transf}).
 
 Also, the purpose of this paper is to develop a phase space analysis for the Cauchy problem \eqref{damped wave equation Heisenberg}. More precisely, our goal is to derive decay estimates on $L^2(\mathbf{H}_n)$ - basis with possible additional $L^1(\mathbf{H}_n)$ - regularity for the Cauchy data as it is stated in the next key theorem.

\begin{theorem}\label{Main Thm}
Let  $(u_0,u_1)\in  H^1(\mathbf{H}_n) \times L^2(\mathbf{H}_n)$ and let $u\in\mathcal{C}([0,\infty), H^1(\mathbf{H}_n))\cap \mathcal{C}^1([0,\infty), L^2(\mathbf{H}_n))$ solve the Cauchy problem \eqref{damped wave equation Heisenberg}. Then, the following estimates are satisfied
\begin{align} \label{estimate u only L2}
\| u(t,\cdot)\|_{L^2(\mathbf{H_n})} & \leq C \left(\|u_0\|_{L^2(\mathbf{H}_n)}+\|u_1\|_{L^2(\mathbf{H}_n)}\right) \\ 
\label{estimate nabla hor u only L2} \| \nabla_{\hor} u(t,\cdot)\|_{L^2(\mathbf{H_n})} & \leq C (1+t)^{-\frac{1}{2}} \left(\|\nabla_{\hor}u_0\|_{L^2(\mathbf{H}_n)}+\|u_1\|_{L^2(\mathbf{H}_n)}\right)  \\ 
\label{estimate partial t u only L2} \| \partial_t u(t,\cdot)\|_{L^2(\mathbf{H_n})} & \leq C (1+t)^{-1} \left(\|\nabla_{\hor}u_0\|_{L^2(\mathbf{H}_n)}+\|u_1\|_{L^2(\mathbf{H}_n)}\right)
\end{align}
for any $t\geq 0$.
Furthermore, if we assume additionally that $u_0,u_1\in L^1(\mathbf{H}_n)$, then, the following decay estimates are satisfied
\begin{align} \label{estimate u}
\| u(t,\cdot)\|_{L^2(\mathbf{H_n})} & \leq C (1+t)^{-\frac{\mathcal{Q}}{4}} \left(\|u_0\|_{L^1(\mathbf{H}_n)}+\|u_0\|_{L^2(\mathbf{H}_n)}+\|u_1\|_{L^1(\mathbf{H}_n)}+\|u_1\|_{L^2(\mathbf{H}_n)}\right) \\ 
\label{estimate nabla hor u} \| \nabla_{\hor} u(t,\cdot)\|_{L^2(\mathbf{H_n})} & \leq C (1+t)^{-\frac{\mathcal{Q}}{4}-\frac{1}{2}} \left(\|u_0\|_{L^1(\mathbf{H}_n)}+\|\nabla_{\hor}u_0\|_{L^2(\mathbf{H}_n)}+\|u_1\|_{L^1(\mathbf{H}_n)}+\|u_1\|_{L^2(\mathbf{H}_n)}\right)  \\ 
\label{estimate partial t u} \| \partial_t u(t,\cdot)\|_{L^2(\mathbf{H_n})} & \leq C (1+t)^{-\frac{\mathcal{Q}}{4}-1} \left(\|u_0\|_{L^1(\mathbf{H}_n)}+\|\nabla_{\hor}u_0\|_{L^2(\mathbf{H}_n)}+\|u_1\|_{L^1(\mathbf{H}_n)}+\|u_1\|_{L^2(\mathbf{H}_n)}\right)
\end{align}
for any $t\geq 0$. Here $C$ is a positive constant.
\end{theorem}

\begin{remark} Let us point out that the estimates in the statement of Theorem \ref{Main Thm} correspond exactly to the $L^2(\mathbb{R}^n)$ - $L^2(\mathbb{R}^n)$ estimates with eventual additional $L^1(\mathbb{R}^n)$ - regularity for the data in the Euclidean case proved in \cite[Lemma 2]{Mat76}, by replacing the dimension $n$ of $\mathbb{R}^n$ with the homogeneous dimension $\mathcal{Q}=2n+2$ of $\mathbf{H}_n$.
\end{remark}

The paper is organized as follows: in Section \ref{Section group fourier transf} and Section \ref{Section Hermit poly}  we will recall some properties of the group Fourier transform on the Heisenberg group and of the Hermite functions, which will play a fundamental role in the paper; then, in Section \ref{Section Proof Main Thm} Theorem \ref{Main Thm} is proved. 

%\paragraph{Notations}

\section{Group Fourier transform on the Heisenberg group} \label{Section group fourier transf}

In this section, we recall the properties of the \emph{group Fourier transform} on $\mathbf{H}_n$ that we will use in order to prove Theorem \ref{Main Thm}. In the following we follow the definitions and the notations of \cite[Chapters 1 and 6]{FR16}. For the general definition of group Fourier transform on locally compact group and its properties we refer to the classical works \cite{Dix77,Dix81,Tay86,CG89,Fol95,Kir04} and references therein.

Let us recall the following equivalent realization of \emph{Schr\"odinger representations} $\{\pi_\lambda\}_{\lambda\in \mathbb{R}^*}$ of $\mathbf{H}_n$ on $L^2(\mathbb{R}^n)$. For any $\lambda\in \mathbb{R}^*$ the mapping $\pi_\lambda$ is a strongly continuous unitary representation defined by
\begin{align*}
\pi_\lambda (x,y,\tau) \varphi (w) \doteq  \mathrm{e}^{i\lambda	\left(\tau +\frac{1}{2}x\cdot y\right)} \mathrm{e}^{i\sign(\lambda)\sqrt{|\lambda |}\, y\cdot w}\varphi\big(w+\sqrt{|\lambda |}x\big)
\end{align*} for any $(x,y,\tau)\in \mathbf{H}_n$, $\varphi\in L^2(\mathbb{R}^n)$ and $w\in \mathbb{R}^n$.

If $f\in L^1(\mathbf{H}_n)$, the group Fourier transform of $f$ is the bounded operator on $L^2(\mathbb{R}^n)$ defined as
\begin{align*}
\pi_\lambda(f) \equiv \widehat{f}(\lambda) \doteq \int_{
\mathbf{H}_n} f(\eta) \, \pi_\lambda(\eta)^* \, \mathrm{d}\eta
\end{align*} for any $\lambda\in \mathbb{R}^*$, that is,
\begin{align*}
(\widehat{f}(\lambda) \varphi_1,\varphi_2)_{L^2(\mathbb{R}^n)} = \int_{\mathbf{H}_n} f(\eta) \, (\pi_\lambda(\eta)^* \varphi_1,\varphi_2)_{L^2(\mathbb{R}^n)}\, \mathrm{d}\eta
\end{align*} for any $\varphi_1,\varphi_2\in L^2(\mathbb{R}^n)$, where $\pi_\lambda(\eta)^*=\pi_\lambda(\eta)^{-1}$ denotes the adjoint operators of the unitary operator $\pi_\lambda(\eta)$. As Schr\"{o}dinger representations are unitary, it follows immediately by the definition the inequality
\begin{equation}\label{Riemann Lebesgue inequality}
\| \widehat{f}(\lambda)\|_{\mathscr{L}(L^2(\mathbb{R}^n)\to L^2(\mathbb{R}^n))}\leq \| f \|_{L^1(\mathbf{H}_n)}
\end{equation} for any $\lambda\in \mathbb{R}^*$.

If $f\in L^2(\mathbf{H}_n)$, then, the following \emph{Plancherel formula} holds
\begin{equation} \label{Plancherel formula}
\| f\|_{L^2(\mathbf{H}_n)}^2 = c_n \int_{\mathbb{R}^*} \| \widehat{f}(\lambda)\|_{\HS [L^2(\mathbb{R}^n)]}^2 \, |\lambda|^n \, \mathrm{d}\lambda,
\end{equation} where $c_n\doteq (2\pi)^{-(3n+1)}$ and $ \| \widehat{f}(\lambda)\|_{\HS [L^2(\mathbb{R}^n)]}$ denotes the Hilbert-Schmidt norm of $\widehat{f}(\lambda)$ (cf. \cite[Proposition 6.2.7]{FR16}). Therefore, for an arbitrary maximal orthonormal system $\{\varphi_k\}_{k\in \mathbb{N}}$ of $L^2(\mathbb{R}^n)$ the definition of Hilbert-Schmidt norm
\begin{align}
 \| \widehat{f}(\lambda)\|_{\HS [L^2(\mathbb{R}^n)]}^2 & \doteq \trace \left((\pi_\lambda(f))^*\,  \pi_\lambda(f)\right)  \notag \\ & = \sum_{k\in \mathbb{N}} \| \widehat{f} (\lambda)\varphi_k \|_{L^2(\mathbb{R}^n)}^2 = \sum_{k,\ell \in \mathbb{N}} \big|( \widehat{f}(\lambda) \varphi_k,\varphi_\ell)_{L^2(\mathbb{R}^n)}\big|^2 \label{Hilbert Schmidt norm}
\end{align} allows us to write \eqref{Plancherel formula} in more operative way, as follows:
\begin{align*} 
\| f\|_{L^2(\mathbf{H}_n)}^2 & = c_n \int_{\mathbb{R}^*} \sum_{k\in \mathbb{N}} \| \widehat{f} (\lambda)\varphi_k \|_{L^2(\mathbb{R}^n)}^2 \, |\lambda|^n \, \mathrm{d}\lambda  = c_n \int_{\mathbb{R}^*} \sum_{k,\ell \in \mathbb{N}} \big|( \widehat{f}(\lambda) \varphi_k,\varphi_\ell)_{L^2(\mathbb{R}^n)}\big|^2 \, |\lambda|^n \, \mathrm{d}\lambda.
\end{align*}
In particular, the Plancherel measure $\mu$ on $\widehat{\mathbf{H}}_n$ is supported on the equivalence classes of $\pi_\lambda$, $\lambda\in\mathbb{R}^*$ and $\mathrm{d}\mu(\pi_\lambda)= c_n |\lambda|^n \mathrm{d}\lambda$.

A further property related to the group Fourier transform, that we will apply extensively, is the action of the \emph{infinitesimal representation} of $\pi_\lambda$ on the generators of the first layer of the Lie algebra $\mathfrak{h}_n$, namely,
\begin{align}
\mathrm{d}\pi_\lambda (X_j) & = \sqrt{|\lambda|} \partial_{w_j} \label{pi lambda Xj} \quad \qquad \quad \, \mbox{for} \ \ j=1,\cdots, n, \\
\mathrm{d}\pi_\lambda (Y_j) & = i\sign(\lambda) \sqrt{|\lambda|} w_j \label{pi lambda Yj} \quad \mbox{for} \ \ j=1,\cdots, n.
\end{align}
In particular, since the action of $\mathrm{d}\pi_\lambda$ can be extended to the universal enveloping algebra of $\mathfrak{h}_n$, combining \eqref{pi lambda Xj} and \eqref{pi lambda Yj}, it follows
\begin{align} \label{pi lambda Delta hor}
\mathrm{d}\pi_\lambda (\Delta_{\hor})= |\lambda| \sum_{j=1}^n (\partial_{w_j}^2-w_j^2) = - |\lambda| \mathrm{H}_w,
\end{align} where $\mathrm{H}_w\doteq  -\Delta_w+|w|^2$ is the \emph{harmonic oscillator} on $\mathbb{R}^n$.

\section{Hermite functions and their properties} \label{Section Hermit poly}

In the last section, we recalled some properties of Schr\"odinger representations. Because of \eqref{pi lambda Delta hor} it will be helpful for what follows to recall the definition of the \emph{Hermite functions} and to prove some elementary properties of them. These functions constitute a complete system of eigenfunctions for the essentially self adjoint operator $\mathrm{H}_w$. 
Let us begin with the definition of $m$-th \emph{Hermite polynomial}
\begin{align*}
H_m(x)\doteq \mathrm{e}^{x^2}\left(\frac{\mathrm{d}}{\mathrm{d} x}\right)^m \mathrm{e}^{-x^2}, \qquad x\in \mathbb{R}.
\end{align*} It is well-known that Hermite polynomials satisfy the the following recurrence properties
\begin{align}
H_{m+1}(x) & =2x H_m(x)-2mH_{m-1}(x) \label{Hermite poly multiplication}, \\
H_{m}'(x) & = 2mH_{m-1}(x) \label{Hermite poly derivative}
\end{align} for any $m\in \mathbb{N}$ and $x\in \mathbb{R}$. Furthermore,
\begin{align*}%\label{Hermite poly orthonarmal relations}
\int_\mathbb{R} H_m(x) H_\ell(x) \mathrm{e}^{-x^2}\,\mathrm{d} x = \sqrt{\pi}\,  2^m m! \, \delta_{m,\ell} \qquad \mbox{for any} \  m,\ell \in \mathbb{N}.
\end{align*} The $m$-th Hermite function is
\begin{align*}
\psi_m(x) \doteq a_m \, \mathrm{e}^{-\frac{x^2}{2}} H_m(x)
\end{align*} for any $m\in \mathbb{N}$ and $x\in\mathbb{R}$, where $a_m\doteq (\sqrt{\pi}\,  2^m m!)^{-1/2}$. By using \eqref{Hermite poly multiplication} and \eqref{Hermite poly derivative} it follows easily that $\psi_m$ is an eigenfunction for the 1-d harmonic oscillator $-\frac{\mathrm{d}^2}{\mathrm{d}x^2}+x^2$ relative to the eigenvalue $2m+1$ as
\begin{align}
-\psi''_m(x)+x^2 \psi_m(x)=(2m+1) \psi_m(x) \label{psi m is eigenfunction}.
\end{align} The multidimensional version of Hermite functions is given by
\begin{align} \label{def e k}
e_{k}(w) \doteq \prod_{j=1}^n \psi_{k_j}(w_j) = \prod_{j=1}^n a_{k_j} H_{k_j}(w_j) \, \mathrm{e}^{-w_j^2/2}
\end{align} for any multi-index $k=(k_j)_{1\leq j\leq n}\in \mathbb{N}^n$ and any $w=(w_j)_{1\leq j\leq n}\in \mathbb{R}^n$. As the elements of $\{e_k\}_{k\in \mathbb{N}^n}$ are functions with separate variables, % by \eqref{Hermite poly orthonarmal relations} we get that $\| e_k\|_{L^2(\mathbb{R}^n)}=1$ and 
by \eqref{psi m is eigenfunction} we get that
\begin{align}\label{e k is eigen function Hw}
\mathrm{H}_w e_k(w) = (-\Delta_w+|w|^2) e_k(w) = \bigg(\sum_{j=1}^n (2k_j+1) \bigg) e_k(w) = \underbrace{(2|k|+n)}_{\doteq \mu_k }e_k(w)
\end{align} for any $k\in \mathbb{N}^n$ and any $w\in \mathbb{R}^n$. 

Let us recall the following result (cf. \cite[Theorem 2.2.3]{NR10}).
\begin{proposition} The system $\{e_k\}_{k\in \mathbb{N}^n}$ is an orthonormal basis of $L^2(\mathbb{R}^n)$.
\end{proposition}

Having in mind \eqref{pi lambda Xj} and \eqref{pi lambda Yj}, we calculate now the action of multiplication and derivation operators on the elements of $\{e_k\}_{k\in \mathbb{N}^n}$.

\begin{lemma} Let $k\in \mathbb{N}^n$ and $1\leq j\leq n$.  Then,
\begin{align}
\partial_{w_j} e_k (w) & =\frac{1}{\sqrt{2}} \left(\sqrt{k\cdot \epsilon_j} \, \, e_{k-\epsilon_j}(w)- e_{k+\epsilon_j}(w)\right) \label{e k derivative} \\
w_j \, e_k(w) & =\frac{1}{\sqrt{2}} \left(\sqrt{k\cdot \epsilon_j} \, \,  e_{k-\epsilon_j}(w)+ e_{k+\epsilon_j}(w)\right) \label{e k multiplication}
\end{align}
for any $w\in \mathbb{R}^n$, where $\epsilon_j=(\delta_{j,h})_{1\leq h\leq n}$ denotes the $j$-th element of the standard basis of $\mathbb{R}^n$.
\end{lemma}
\begin{proof}
Let us begin with \eqref{e k derivative}. By the definition \eqref{def e k} of multidimensional Hermite function it follows
\begin{align*}
\partial_{w_j} e_k (w) & =  \partial_{w_j} \left( a_{k_j} H_{k_j}(w_j) \, \mathrm{e}^{-w_j^2/2}\right) \prod_{\substack{h = 1\\  h \not = j}}^n a_{k_h} H_{k_h}(w_h) \, \mathrm{e}^{-w_h^2/2}.  
\end{align*} Applying \eqref{Hermite poly multiplication} and \eqref{Hermite poly derivative}, we find 
\begin{align*}
\partial_{w_j} \Big( a_{k_j} H_{k_j}(w_j) \, \mathrm{e}^{-w_j^2/2}\Big) & = a_{k_j} \Big( H'_{k_j}(w_j) -w_j \, H_{k_j}(w_j) \Big)  \mathrm{e}^{-w_j^2/2} \\
& = a_{k_j} \Big( k_j H_{k_j-1}(w_j) -\tfrac{1}{2} \, H_{k_j+1}(w_j)\Big)  \mathrm{e}^{-w_j^2/2} .
\end{align*} Therefore,
\begin{align*}
\partial_{w_j} e_k (w) & =  a_{k_j} \Big( k_j H_{k_j-1}(w_j) -\tfrac{1}{2} \, H_{k_j+1}(w_j)\Big)  \mathrm{e}^{-w_j^2/2}  \prod_{\substack{h=1\\  h\neq j}}^n a_{k_h} H_{k_h}(w_h) \, \mathrm{e}^{-w_h^2/2}  \\
& =  \frac{a_{k_j}k_j}{a_{k_j-1}}  \Big(a_{k_j-1} H_{k_j-1}(w_j) \mathrm{e}^{-w_j^2/2} \Big) \prod_{\substack{h=1\\  h\neq j}}^n a_{k_h} H_{k_h}(w_h) \, \mathrm{e}^{-w_h^2/2}  \\ & \quad -  \frac{a_{k_j}}{2 a_{k_j+1}} \Big(  a_{k_j +1} H_{k_j+1}(w_j) \mathrm{e}^{-w_j^2/2} \Big) \prod_{\substack{h=1\\  h\neq j}}^n a_{k_h} H_{k_h}(w_h) \, \mathrm{e}^{-w_h^2/2}  = \frac{a_{k_j}k_j}{a_{k_j-1}}    e_{k-\epsilon_j}(w)-  \frac{a_{k_j}}{2 a_{k_j+1}} e_{k+\epsilon_j}.
\end{align*} Since $a_{k_j}k_j/a_{k_j-1}= \sqrt{k_j/2}$ and $a_{k_j}/2 a_{k_j+1}= 1/\sqrt{2}$, the last equality yields \eqref{e k derivative}. Similarly, by
\begin{align*}
w_j  H_{k_j}(w_j) \, \mathrm{e}^{-w_j^2/2} =   \Big( k_j H_{k_j-1}(w_j) +\tfrac{1}{2} \, H_{k_j+1}(w_j)\Big)  \mathrm{e}^{-w_j^2/2}
\end{align*} we get \eqref{e k multiplication}. The proof is complete.
\end{proof}

\section{Proof of Theorem \ref{Main Thm}} \label{Section Proof Main Thm}

In this section, we will prove Theorem \ref{Main Thm}. We split the proof in three parts for the estimate of the $L^2(\mathbf{H}_n)$ - norm of $u(t,\cdot), \partial_t u(t,\cdot)$ and $\nabla_{\hor}u(t,\cdot)$, respectively. 

The plan is to adapt the main ideas of WKB analysis from the pioneering work of Matsumura \cite{Mat76} to the case of the Heisenberg group. Since the group Fourier transform of a $L^2(\mathbf{H}_n)$ function $f$ is no longer a function but a family of bounded linear operators $\{\widehat{f}(\lambda)\}_{\lambda\in\mathbb{R}^*}$ on $L^2(\mathbb{R}^n)$, the trick is to project somehow these  operators by using the basis $\{e_k\}_{k\in\mathbb{N}^n}$, namely, by working with $$\big\{(\widehat{f}(\lambda) e_k,e_\ell)_{L^2(\mathbb{R}^n)}\big\}_{k,\ell\in \mathbb{N}^n}.$$ 

This approach has been introduced in \cite{RT18}. However, differently from \cite[Section 2]{RT18}, where the combined presence of a mass and a damping term produces an exponential decay  for the $L^2(\mathbf{H}_n)$ - norm of the solution and of all its derivatives, in the massless case for each component we will split the corresponding integral term coming from Plancherel formula in two zones with respect to $\lambda$. On the one hand, for $|\lambda|$  \textquotedblleft small\textquotedblright $\,$ the analogous of Riemann-Lebesgue inequality in \eqref{Riemann Lebesgue inequality} is employed in order to get the decay rates of polynomial order as in \eqref{estimate u}, \eqref{estimate nabla hor u} and \eqref{estimate partial t u}.   On the other hand, for $|\lambda|$  \textquotedblleft large\textquotedblright $\, $ by Plancherel formula we get exponential decay under suitable regularity assumptions for the Cauchy data, provided that suitable $L^2$ regularity is required for the Cauchy data. Let us stress that this splitting depends on the multi-index $k\in\mathbb{N}^n$ and the corresponding eigenvalue $\mu_k$ of $\mathrm{H}_w$. In this way, we will show the validity of \eqref{estimate u}, \eqref{estimate nabla hor u} and \eqref{estimate partial t u} first. As byproduct of the above described procedure we get immediately estimates \eqref{estimate u only L2}, \eqref{estimate nabla hor u only L2} and \eqref{estimate partial t u only L2} too, if we use Plancherel formula for  \textquotedblleft small\textquotedblright $\,$ $|\lambda|$ as well, rather than the Riemann-Lebesgue inequality.

\subsection{Estimate of the $L^2(\mathbf{H}_n)$ - norm of $u(t,\cdot)$}

Let us consider $u$ solution to \eqref{damped wave equation Heisenberg}. By acting by the group Fourier transform on \eqref{damped wave equation Heisenberg} we get a Cauchy problem related to a parameter dependent functional differential equation for $\widehat{u}(t,\lambda)$
\begin{align}\label{damped wave equation Heisenberg transform}
\begin{cases}
\partial_t^2 \widehat{u}(t,\lambda)+\partial_t \widehat{u}(t,\lambda)-\sigma_{\Delta_{\hor}}(\lambda) \, \widehat{u}(t,\lambda)= 0, & \lambda\in \mathbb{R}^*, \ t>0, \\
\widehat{u}(0,\lambda)= \widehat{u}_0(\lambda), & \lambda \in \mathbb{R}^*, \\
\partial_t \widehat{u}(0,\lambda)= \widehat{u}_1(\lambda), & \lambda\in \mathbb{R}^*,
\end{cases}
\end{align} where $\sigma_{\Delta_{\hor}}(\lambda) $ is the symbol of the sub-Laplacian. Due to \eqref{pi lambda Delta hor}, we have $\sigma_{\Delta_{\hor}}(\lambda)= -|\lambda| \mathrm{H}_w$. Let us introduce the notation
\begin{align}\label{def u hat projections}
\widehat{u}(t,\lambda)_{k,\ell}\doteq \big(\widehat{u}(t,\lambda) e_k,e_\ell\big)_{L^2(\mathbb{R}^n)} \qquad \mbox{for any} \ \ k,\ell\in \mathbb{N}^n,
\end{align} where $\{e_k\}_{k\in \mathbb{N}^n}$ is the system of Hermite functions.
 Since $\mathrm{H}_w e_k= \mu_k e_k$, then, $\widehat{u}(t,\lambda)_{k,\ell}$ solves an ordinary differential equation depending on parameters $\lambda\in\mathbb{R}^*$ and $k,\ell\in \mathbb{N}^n$
\begin{align}\label{damped wave equation Heisenberg project}
\begin{cases}
\partial_t^2 \widehat{u}(t,\lambda)_{k,\ell}+\partial_t \widehat{u}(t,\lambda)_{k,\ell}+\mu_k |\lambda| \widehat{u}(t,\lambda)_{k,\ell}= 0, &   t>0, \\
\widehat{u}(0,\lambda)_{k,\ell}= \widehat{u}_0(\lambda)_{k,\ell},  \\
\partial_t \widehat{u}(0,\lambda)_{k,\ell}= \widehat{u}_1(\lambda)_{k,\ell},
\end{cases}
\end{align} where $$\widehat{u}_h(\lambda)_{k,\ell}\doteq \big(\widehat{u}_h(\lambda) e_k,e_\ell\big)_{L^2(\mathbb{R}^n)} \qquad \mbox{for any} \ h=0,1 \ \mbox{and any} \ k,\ell\in \mathbb{N}^n.$$ The roots of the characteristic equation $\tau^2+\tau+\mu_k|\lambda|=0$  are
\begin{align*}
\tau_{\pm} =\begin{cases} -\frac{1}{2}\pm i \sqrt{\mu_k |\lambda|-\frac{1}{4}} & \mbox{if} \ \  4\mu_k |\lambda| >1 , \\
-\frac{1}{2} & \mbox{if} \ \  4\mu_k |\lambda| =1 , \\
-\frac{1}{2}\pm  \sqrt{\frac{1}{4}-\mu_k |\lambda|} & \mbox{if} \ \  4\mu_k |\lambda| <1 . 
\end{cases} 
\end{align*} Elementary computations lead to the representation formula
\begin{align} \label{representation u hat components}
\widehat{u}(t,\lambda)_{k,\ell} = \widehat{u}_0(\lambda)_{k,\ell}\, \mathrm{e}^{-\frac{t}{2}} F(t, \lambda,k)+\left(\widehat{u}_0(\lambda)_{k,\ell}+\tfrac{1}{2}\widehat{u}_1(\lambda)_{k,\ell}\right)\mathrm{e}^{-\frac{t}{2}} G(t, \lambda,k),
\end{align} where 
\begin{align*}
F(t, \lambda,k) &\doteq  \begin{cases} \cos\left( \sqrt{\mu_k |\lambda|-\frac{1}{4}} \, t\right) & \mbox{if} \ \  4\mu_k |\lambda| >1 , \\
1 & \mbox{if} \ \  4\mu_k |\lambda| =1 , \\
\cosh\left( \sqrt{\frac{1}{4}-\mu_k |\lambda|} \, t\right)  & \mbox{if} \ \  4\mu_k |\lambda| <1 ,
\end{cases}  \\
G(t, \lambda,k) &\doteq \begin{cases}\dfrac{ \sin\left( \sqrt{\mu_k |\lambda|-\frac{1}{4}} \, t\right)}{\sqrt{\mu_k |\lambda|-\frac{1}{4}}} & \mbox{if} \ \  4\mu_k |\lambda| >1 , \\
t & \mbox{if} \ \  4\mu_k |\lambda| =1 , \\
\dfrac{\sinh\left( \sqrt{\frac{1}{4}-\mu_k |\lambda|} \, t\right) }{ \sqrt{\frac{1}{4}-\mu_k |\lambda|}} & \mbox{if} \ \  4\mu_k |\lambda| <1 .
\end{cases}
\end{align*} Note that $F(t,\lambda,k)=\partial _t G(t, \lambda,k)$ for any $\lambda\in \mathbb{R}^*$. 
By using Plancherel formula and $\{e_k\}_{k\in\mathbb{N}^n}$ as orthonormal basis of $L^2(\mathbb{R}^n)$, we have 
\begin{align*}
\| u(t,\cdot) & \|_{L^2(\mathbf{H}_n)}^2  \\ & = c_n \int_{\mathbb{R}^*}\| \widehat{u}(t,\lambda)\|^2_{\HS [L^2(\mathbb{R}^n)]} \, |\lambda|^n \,\mathrm{d}\lambda = c_n \sum_{k,\ell\in\mathbb{N}^n} \int_{\mathbb{R}^*}\big( \widehat{u}(t,\lambda) e_k,e_\ell\big)^2_{L^2(\mathbb{R}^n)}  |\lambda|^n \,\mathrm{d}\lambda \\
& = c_n \sum_{k,\ell \in\mathbb{N}^n} \left( \int_{0<|\lambda|<\frac{1}{8\mu_k}} \big|\big( \widehat{u}(t,\lambda) e_k,e_\ell\big)_{L^2(\mathbb{R}^n)}\big|^2  |\lambda|^n \,\mathrm{d}\lambda + \int_{|\lambda| >\frac{1}{8\mu_k}} \big|\big( \widehat{u}(t,\lambda) e_k,e_\ell\big)_{L^2(\mathbb{R}^n)}\big|^2 |\lambda|^n \,\mathrm{d}\lambda \right)  \\ & \doteq I^{\mathrm{low}}+ I^{\mathrm{high}}.
\end{align*} We will estimate separately the terms $I^{\mathrm{low}}$ and $ I^{\mathrm{high}}$. We start with
\begin{align*}
 I^{\mathrm{high}} =  c_n \sum_{k,\ell\in\mathbb{N}^n} \int_{|\lambda| >\frac{1}{8\mu_k}} | \widehat{u}(t,\lambda)_{k,\ell}|^2\, |\lambda|^n \,\mathrm{d}\lambda.
\end{align*} For any $\lambda$ such that $|\lambda| > 1/ (8\mu_k)$ we may estimate
\begin{align*}
\mathrm{e}^{-\frac{t}{2}} |F(t,\lambda,k) |,  \mathrm{e}^{-\frac{t}{2}}|G(t,\lambda,k)| \lesssim \mathrm{e}^{-\frac{\delta}{2}t}
\end{align*} for some $\delta>0$, where $\delta$ and the unexpressed multiplicative constant hereafter are independent of the time variable and of the parameters $\lambda$ and $k,\ell$ as well. Then, by the representation formula \eqref{representation u hat components} we obtain 
\begin{align}\label{estimate u hat k,l}
 |\widehat{u}(t,\lambda)_{k,\ell}|^2 \lesssim  \mathrm{e}^{-\delta t} \left(|\widehat{u}_0(\lambda)_{k,\ell}|^2+|\widehat{u}_1(\lambda)_{k,\ell}|^2\right) 
\end{align} for any $|\lambda| >1/(8\mu_k)$. Therefore,
\begin{align}
 I^{\mathrm{high}} & \lesssim   \mathrm{e}^{-\delta t}  \sum_{k,\ell\in\mathbb{N}^n} \int_{|\lambda| >\frac{1}{8\mu_k}}\left(|\widehat{u}_0(\lambda)_{k,\ell}|^2+|\widehat{u}_1(\lambda)_{k,\ell}|^2\right)   |\lambda|^n \,\mathrm{d}\lambda \notag \\
&  \lesssim   \mathrm{e}^{-\delta t}  \int_{\mathbb{R}^*} \sum_{k,\ell\in\mathbb{N}^n} \left(|\widehat{u}_0(\lambda)_{k,\ell}|^2+|\widehat{u}_1(\lambda)_{k,\ell}|^2\right)   |\lambda|^n \,\mathrm{d}\lambda \notag \\
& =  \mathrm{e}^{-\delta t}  \int_{\mathbb{R}^*}  \left(\|\widehat{u}_0(\lambda)\|^2_{\HS[L^2(\mathbb{R}^n)]}+\|\widehat{u}_1(\lambda)\|^2_{\HS[L^2(\mathbb{R}^n)]}\right)   |\lambda|^n \,\mathrm{d}\lambda \approx \mathrm{e}^{-\delta t} \Big(\|u_0\|_{L^2(\mathbf{H}_n)}^2+\|u_1\|_{L^2(\mathbf{H}_n)}^2\Big), \label{estimate I high}
\end{align} where in the last step we applied Plancherel formula to $u_0$ and $u_1$. 

\begin{remark}\label{remark u}  Let us point out that if we do not restrict to the range of $\lambda$ such that $|\lambda|>1/(8\mu_k)$, then one has to replace \eqref{estimate u hat k,l} by 
\begin{align*}
 |\widehat{u}(t,\lambda)_{k,\ell}|^2 \lesssim  \left(|\widehat{u}_0(\lambda)_{k,\ell}|^2+|\widehat{u}_1(\lambda)_{k,\ell}|^2\right) 
\end{align*} since $|F(t,\lambda,k)|,|G(t,\lambda,k)|\to 1$ as $|\lambda|\to 0$. So, we would end up with the uniform estimate 
\begin{align}\label{boundedness L2 norm u(t)}
\| u(t,\cdot)  \|_{L^2(\mathbf{H}_n)}^2 \lesssim \|u_0\|_{L^2(\mathbf{H}_n)}^2+\|u_1\|_{L^2(\mathbf{H}_n)}^2
\end{align} with no time-dependent decay function on the right-hand side. Clearly, \eqref{boundedness L2 norm u(t)} implies immediately \eqref{estimate u only L2} in the statement of Theorem \ref{Main Thm}.
\end{remark}

We estimate now the other term 
\begin{align*}
 I^{\mathrm{low}} =  c_n \sum_{k,\ell\in\mathbb{N}^n} \int_{0<|\lambda| <\frac{1}{8\mu_k}} | \widehat{u}(t,\lambda)_{k,\ell}|^2\, |\lambda|^n \,\mathrm{d}\lambda.
\end{align*} For $\lambda$ such that $|\lambda|<1/(8\mu_k)$ it holds
\begin{align*}
\mathrm{e}^{-\frac{t}{2}} |F(t,\lambda,k) | & = \mathrm{e}^{-\frac{t}{2}} \cosh\left( \sqrt{\tfrac{1}{4}-\mu_k |\lambda|} \, t\right) \lesssim \mathrm{e}^{\left(-\frac{1}{2}+\frac{1}{2}\sqrt{1-4\mu_k|\lambda|}\right)t}, 
\\  \mathrm{e}^{-\frac{t}{2}}|G(t,\lambda,k)| & =  \frac{ \mathrm{e}^{-\frac{t}{2}} \sinh\left( \sqrt{\frac{1}{4}-\mu_k |\lambda|} \, t\right) }{ \sqrt{\frac{1}{4}-\mu_k |\lambda|}}  \lesssim \mathrm{e}^{\left(-\frac{1}{2}+\frac{1}{2}\sqrt{1-4\mu_k|\lambda|}\right)t}.
\end{align*} Using the inequality 
\begin{align}\label{sqrt(1-4z) estimate}
-4z\leq -1+\sqrt{1-4z}\leq -2z \qquad \mbox{for any} \ z\geq 0,
\end{align} it follows that
\begin{align*}
 |\widehat{u}(t,\lambda)_{k,\ell}|^2 & \lesssim \mathrm{e}^{-t+\sqrt{1-4\mu_k|\lambda|}\, t} \left(|\widehat{u}_0(\lambda)_{k,\ell}|^2+|\widehat{u}_1(\lambda)_{k,\ell}|^2\right)  \\ &\lesssim \mathrm{e}^{- 2 \mu_k|\lambda| t} \left(|\widehat{u}_0(\lambda)_{k,\ell}|^2+|\widehat{u}_1(\lambda)_{k,\ell}|^2\right) , 
\end{align*}  for any $\lambda$ such that $|\lambda|<1/(8\mu_k)$. Applying the last estimate, we find
\begin{align*}
 I^{\mathrm{low}} & \lesssim \sum_{k,\ell\in\mathbb{N}^n} \int_{0<|\lambda| <\frac{1}{8\mu_k}} \mathrm{e}^{- 2 \mu_k|\lambda| t} \left(|\widehat{u}_0(\lambda)_{k,\ell}|^2+|\widehat{u}_1(\lambda)_{k,\ell}|^2\right)   |\lambda|^n \,\mathrm{d}\lambda \\
 & = \sum_{k\in\mathbb{N}^n} \int_{0<|\lambda| <\frac{1}{8\mu_k}} \mathrm{e}^{- 2 \mu_k|\lambda| t}  \left( \ \sum_{\ell\in\mathbb{N}^n} \big|\big(\widehat{u}_0(\lambda)e_k, e_\ell\big)_{L^2(\mathbb{R}^n)}\big|^2+\big|\big(\widehat{u}_1(\lambda)e_k, e_\ell\big)_{L^2(\mathbb{R}^n)}\big|^2\right)   |\lambda|^n \,\mathrm{d}\lambda \\
 & = \sum_{k\in\mathbb{N}^n} \int_{0<|\lambda| <\frac{1}{8\mu_k}} \mathrm{e}^{- 2 \mu_k|\lambda| t}  \left(\|\widehat{u}_0(\lambda)e_k\|_{L^2(\mathbb{R}^n)}^2+\|\widehat{u}_1(\lambda)e_k\|_{L^2(\mathbb{R}^n)}^2\right)   |\lambda|^n \,\mathrm{d}\lambda \\
 & \lesssim \sum_{k\in\mathbb{N}^n} \int_{0}^{\frac{1}{8\mu_k}} \mathrm{e}^{- 2 \mu_k \lambda t}    \lambda^n \,\mathrm{d}\lambda \   \Big(\|u_0\|_{L^1(\mathbf{H}_n)}^2+\|u_1\|_{L^1(\mathbf{H}_n)}^2\Big),
\end{align*} where in the third step we employed Parseval's identity
\begin{align}\label{Parseval's identity}
\|\widehat{u}_h(\lambda)e_k\|_{L^2(\mathbb{R}^n)}^2=\sum_{\ell\in\mathbb{N}^n} \big|\big(\widehat{u}_0(\lambda)e_k, e_\ell\big)_{L^2(\mathbb{R}^n)}\big|^2 \qquad \mbox{for} \ h=0,1
\end{align}
 and in the last step we used \eqref{Riemann Lebesgue inequality} and $\|e_k\|_{L^2(\mathbb{R}^n)}=1$ to get the inequality 
\begin{align} \label{estimate u hat e k}
\|\widehat{u}_h(\lambda)e_k\|_{L^2(\mathbb{R}^n)} & \leq \|\widehat{u}_h(\lambda)\|_{\mathscr{L}\left(L^2(\mathbb{R}^n)\to L^2(\mathbb{R}^n)\right)}\|e_k\|_{L^2(\mathbb{R}^n)} \leq \|u_h\|_{L^1(\mathbf{H}_n)} \quad \mbox{for} \ h=0,1.
\end{align} Carrying out the change of variables $\theta= 2\mu_k \lambda t$ in the last integral, we have
\begin{align}
 I^{\mathrm{low}} & \lesssim \sum_{k\in\mathbb{N}^n} (2\mu_k t)^{-(n+1)} \int_{0}^{\frac{t}{4}} \mathrm{e}^{- \theta}    \theta^n \,\mathrm{d}\theta \   \Big(\|u_0\|_{L^1(\mathbf{H}_n)}^2+\|u_1\|_{L^1(\mathbf{H}_n)}^2\Big) \notag \\
 & \lesssim  \Gamma(n+1)\,  t^{-(n+1)}\left(\,  \sum_{k\in\mathbb{N}^n} \mu_k ^{-(n+1)} \right)   \Big(\|u_0\|_{L^1(\mathbf{H}_n)}^2+\|u_1\|_{L^1(\mathbf{H}_n)}^2\Big)\notag  \\
 & \lesssim    t^{-\frac{\mathcal{Q}}{2}}  \Big(\|u_0\|_{L^1(\mathbf{H}_n)}^2+\|u_1\|_{L^1(\mathbf{H}_n)}^2\Big), \label{estimate I low}
\end{align} where $\Gamma$ denotes Euler integral of the second kind  and  in the last estimate we used that the series
 \begin{align}\label{series convergent}
  \sum_{k\in\mathbb{N}^n} \mu_k ^{-(n+1)}  = \sum_{k\in\mathbb{N}^n} (2|k|+n)^{-(n+1)}
 \end{align} is convergent.
Combining \eqref{estimate I high}, \eqref{estimate I low} and \eqref{boundedness L2 norm u(t)} (which allows us to exclude a singular behavior of $I^{\mathrm{low}} $ as $t\to 0^+$), we get finally \eqref{estimate u}.

\subsection{Estimate of the $L^2(\mathbf{H}_n)$ - norm of $\partial_t u(t,\cdot)$}

In this section we estimate the $L^2(\mathbf{H}_n)$ - norm of the time derivative of $u$. As in the case of  the $\|u(t,\cdot)\|_{L^2(\mathbf{H}_n)}$, we will split the estimate in two parts. By Plancherel formula we have
\begin{align*}
\| \partial_t u(t,\cdot)  \|_{L^2(\mathbf{H}_n)}^2  & = c_n \int_{\mathbb{R}^*}\| \partial_t \widehat{u}(t,\lambda)\|^2_{\HS [L^2(\mathbb{R}^n)]} \, |\lambda|^n \,\mathrm{d}\lambda  \\
& = c_n \sum_{k,\ell \in\mathbb{N}^n} \left( \int_{0<|\lambda|<\frac{1}{8\mu_k}}  + \int_{|\lambda| >\frac{1}{8\mu_k}}   \right) \big|\big( \partial_t\widehat{u}(t,\lambda) e_k,e_\ell\big)_{L^2(\mathbb{R}^n)}\big|^2 |\lambda|^n \,\mathrm{d}\lambda  \\
& = c_n \sum_{k,\ell \in\mathbb{N}^n} \left( \int_{0<|\lambda|<\frac{1}{8\mu_k}}  + \int_{|\lambda| >\frac{1}{8\mu_k}}   \right) \big| \partial_t\widehat{u}(t,\lambda)_{ k,\ell}\big|^2 |\lambda|^n \,\mathrm{d}\lambda \doteq J^{\mathrm{low}}+ J^{\mathrm{high}}.
\end{align*}

We estimate $J^{\mathrm{low}}$ first. For $|\lambda|<1/(8\mu_k)$ by \eqref{representation u hat components} we find
\begin{align}
\partial_t \widehat{u}(t,\lambda)_{ k,\ell} & = -\frac{ \mathrm{e}^{-\frac{t}{2}}\sinh \left(\sqrt{\frac{1}{4}-\mu_k|\lambda|} \, t\right) }{\sqrt{\frac{1}{4}-\mu_k|\lambda|}} \, \mu_k |\lambda|  \widehat{u}_0(\lambda)_{ k,\ell} \notag \\
& \quad + \mathrm{e}^{-\frac{t}{2}} \left(\cosh \left(\sqrt{\tfrac{1}{4}-\mu_k|\lambda|} \, t\right) -\frac{ \sinh \left(\sqrt{\frac{1}{4}-\mu_k|\lambda|} \, t\right)}{2\sqrt{\frac{1}{4}-\mu_k|\lambda|}} \right)  \widehat{u}_1(\lambda)_{ k,\ell}  \label{representation formula du/dt elliptic zone} 
\end{align} and, rewriting the factor that multiplies $\widehat{u}_1(\lambda)_{ k,\ell}$, we arrive at
\begin{align*}
 \partial_t \widehat{u}(t,\lambda)_{ k,\ell} & = -\frac{ \mathrm{e}^{-\frac{t}{2}}\sinh \left(\sqrt{\frac{1}{4}-\mu_k|\lambda|} \, t\right) }{\sqrt{\frac{1}{4}-\mu_k|\lambda|}} \, \mu_k |\lambda|  \widehat{u}_0(\lambda)_{ k,\ell} \notag \\
& \quad + \mathrm{e}^{-\frac{t}{2}} \! \left(\!\left(1- \frac{ 1}{2\sqrt{\frac{1}{4}-\mu_k|\lambda|}} \right) \mathrm{e}^{
\frac{1}{2}\sqrt{1-4\mu_k|\lambda|} \, t} +\left(1+ \frac{ 1}{2\sqrt{\frac{1}{4}-\mu_k|\lambda|}} \right) \mathrm{e}^{
-\frac{1}{2}\sqrt{1-4\mu_k|\lambda|} \, t} \right)  \widehat{u}_1(\lambda)_{ k,\ell}.  \notag
\end{align*}
Combining the last expression and \eqref{sqrt(1-4z) estimate}, we may estimate for $|\lambda|<1/(8\mu_k)$
\begin{align*}
|\partial_t \widehat{u}(t,\lambda)_{ k,\ell}|  &\lesssim \mathrm{e}^{-\frac{t}{2}+\frac{1}{2}\sqrt{1-4\mu_k|\lambda|} \, t} \mu_k |\lambda| \left( | \widehat{u}_0(\lambda)_{ k,\ell}|+ | \widehat{u}_1(\lambda)_{ k,\ell}|\right)+ \mathrm{e}^{-\frac{t}{2}-\frac{1}{2}\sqrt{1-4\mu_k|\lambda|} \, t}  | \widehat{u}_1(\lambda)_{ k,\ell}|  \\
& \lesssim \mathrm{e}^{-\mu_k|\lambda|  t} \mu_k |\lambda| \left( | \widehat{u}_0(\lambda)_{ k,\ell}|+ | \widehat{u}_1(\lambda)_{ k,\ell}|\right)+ \mathrm{e}^{-\frac{t}{2}}  | \widehat{u}_1(\lambda)_{ k,\ell}|.
\end{align*}
Therefore, combining the last inequality and \eqref{Parseval's identity}, it results
\begin{align}
J^{\mathrm{low}} & = c_n \sum_{k,\ell \in\mathbb{N}^n} \int_{0<|\lambda|<\frac{1}{8\mu_k}}  \big| \partial_t\widehat{u}(t,\lambda)_{ k,\ell}\big|^2 |\lambda|^n \,\mathrm{d}\lambda \notag \\
& \lesssim  \sum_{k,\ell \in\mathbb{N}^n} \mu_k^2 \int_{0<|\lambda|<\frac{1}{8\mu_k}} \mathrm{e}^{-2\mu_k|\lambda|  t}   \left( | \widehat{u}_0(\lambda)_{ k,\ell}|^2+ | \widehat{u}_1(\lambda)_{ k,\ell}|^2\right) |\lambda|^{n+2} \,\mathrm{d}\lambda  \notag \\ 
& \qquad + \mathrm{e}^{-t} \sum_{k,\ell \in\mathbb{N}^n} \int_{0<|\lambda|<\frac{1}{8\mu_k}}  | \widehat{u}_1(\lambda)_{ k,\ell}|^2 |\lambda|^n \,\mathrm{d}\lambda \notag \\
& = \sum_{k \in\mathbb{N}^n} \mu_k^2 \int_{0<|\lambda|<\frac{1}{8\mu_k}} \mathrm{e}^{-2\mu_k|\lambda|  t}    \left( \, \sum_{\ell \in\mathbb{N}^n} \big| \big(\widehat{u}_0(\lambda) e_k, e_{\ell} \big)_{L^2(\mathbb{R}^n)} \big|^2+ \big| \big(\widehat{u}_1(\lambda) e_k, e_{\ell} \big)_{L^2(\mathbb{R}^n)} \big|^2 \right) |\lambda|^{n+2} \,\mathrm{d}\lambda \notag \\ 
& \qquad + \mathrm{e}^{-t} \sum_{k \in\mathbb{N}^n} \int_{0<|\lambda|<\frac{1}{8\mu_k}} \sum_{\ell \in\mathbb{N}^n} \big| \big(\widehat{u}_1(\lambda) e_k, e_{\ell} \big)_{L^2(\mathbb{R}^n)} \big|^2 |\lambda|^n \,\mathrm{d}\lambda \notag\\
& = \sum_{k \in\mathbb{N}^n} \mu_k^2 \int_{0<|\lambda|<\frac{1}{8\mu_k}} \mathrm{e}^{-2\mu_k|\lambda|  t}    \left( \|\widehat{u}_0(\lambda) e_k\|_{L^2(\mathbb{R}^n)}^2+ \|\widehat{u}_1(\lambda) e_k\|_{L^2(\mathbb{R}^n)}^2 \right) |\lambda|^{n+2} \,\mathrm{d}\lambda \notag \\ 
& \qquad + \mathrm{e}^{-t} \sum_{k \in\mathbb{N}^n} \int_{0<|\lambda|<\frac{1}{8\mu_k}} \|\widehat{u}_1(\lambda) e_k\|_{L^2(\mathbb{R}^n)}^2 |\lambda|^n \,\mathrm{d}\lambda. \label{estimate for the remark}
\end{align} So, applying the inequality in \eqref{estimate u hat e k} and using the fact that the series \eqref{series convergent} converges, it follows
\begin{align}
J^{\mathrm{low}} \lesssim & \sum_{k \in\mathbb{N}^n} \mu_k^2 \int_{0}^{\frac{1}{8\mu_k}} \mathrm{e}^{-2\mu_k\lambda  t}    \lambda^{n+2} \,\mathrm{d}\lambda \,  \left( \|u_0\|_{L^1(\mathbf{H}_n)}^2+ \|u_1\|_{L^1(\mathbf{H}_n)}^2 \right) + \mathrm{e}^{-t} \sum_{k \in\mathbb{N}^n} \int_{0}^{\frac{1}{8\mu_k}} \lambda^n \,\mathrm{d}\lambda \,  \|u_1\|_{L^1(\mathbf{H}_n)}^2\notag  \\
\lesssim & \,  t^{-(n+3)} \sum_{k \in\mathbb{N}^n} \mu_k^{-(n+1)} \int_{0}^{\frac{t}{4}} \mathrm{e}^{-\theta}    \theta^{n+2} \,\mathrm{d}\theta \,  \left( \|u_0\|_{L^1(\mathbf{H}_n)}^2+ \|u_1\|_{L^1(\mathbf{H}_n)}^2 \right) + \mathrm{e}^{-t} \sum_{k \in\mathbb{N}^n}  \mu_k^{-(n+1)}  \,  \|u_1\|_{L^1(\mathbf{H}_n)}^2\notag \\
\lesssim & \,  t^{-\frac{\mathcal{Q}}{2}-2}    \left( \|u_0\|_{L^1(\mathbf{H}_n)}^2+ \|u_1\|_{L^1(\mathbf{H}_n)}^2 \right) + \mathrm{e}^{-t}\,    \|u_1\|_{L^1(\mathbf{H}_n)}^2,\label{estimate J low}
\end{align} where in second inequality we performed the change of variables $\theta= 2\mu_k \lambda t$ in the first integral.
Now we will prove that $J^{\mathrm{high}}$ decay with exponential speed with respect $t$. For $\lambda$ such that $1/(8\mu_k)<|\lambda|<1/(4\mu_k)$ by \eqref{representation formula du/dt elliptic zone} we get 
\begin{align} \label{intermediate estimate du/dt}
 \big| \partial_t\widehat{u}(t,\lambda)_{ k,\ell}\big| \lesssim \mathrm{e}^{-\frac{\delta}{2}t} \big(| \widehat{u}_0(\lambda)_{ k,\ell}|+ | \widehat{u}_1(\lambda)_{ k,\ell}|\big)
\end{align} for some $\delta>0$, that may change from line to line in the following. When $|\lambda|=1/(4\mu_k)$, by \eqref{representation u hat components} we have
\begin{align}
\partial_t\widehat{u}(t,\lambda)_{ k,\ell} = \mathrm{e}^{-\frac{t}{2}}   \widehat{u}_1(\lambda)_{ k,\ell}  - \tfrac{1}{2} \mathrm{e}^{-\frac{t}{2}} t \, \big(  \widehat{u}_1(\lambda)_{ k,\ell} +\tfrac{1}{2}\widehat{u}_0(\lambda)_{ k,\ell} \big) \label{representation formula du/dt separation line} 
\end{align}
while for $\lambda$ such that $1/(4\mu_k)<|\lambda|<1/(2\mu_k)$
\begin{align}
\partial_t \widehat{u}(t,\lambda)_{ k,\ell} & = -\frac{ \mathrm{e}^{-\frac{t}{2}}\sin \left(\sqrt{\mu_k|\lambda|-\frac{1}{4}} \, t\right) }{\sqrt{\mu_k|\lambda|-\frac{1}{4}}} \, \mu_k |\lambda|  \widehat{u}_0(\lambda)_{ k,\ell} \notag \\
& \quad + \mathrm{e}^{-\frac{t}{2}} \left(\cos \left(\sqrt{\mu_k|\lambda|-\frac{1}{4}} \, t\right) -\frac{ \sin \left(\sqrt{\mu_k|\lambda|-\frac{1}{4}} \, t\right)}{2\sqrt{\mu_k|\lambda|-\frac{1}{4}}} \right)  \widehat{u}_1(\lambda)_{ k,\ell}  \label{representation formula du/dt hyperbolic zone} 
\end{align} In the case $1/(4\mu_k)\leq |\lambda|<1/(2\mu_k)$ from \eqref{representation formula du/dt separation line} and \eqref{representation formula du/dt hyperbolic zone} we find again the estimate \eqref{intermediate estimate du/dt}. Finally, for $|\lambda| >1/(2\mu_k)$  we have to estimate
\begin{align*}
 \big| \partial_t\widehat{u}(t,\lambda)_{ k,\ell}\big| \lesssim \mathrm{e}^{-\frac{t}{2}} \big(\sqrt{\mu_k |\lambda|} \,| \widehat{u}_0(\lambda)_{ k,\ell}|+ | \widehat{u}_1(\lambda)_{ k,\ell}|\big).
\end{align*} Hence, applying the last estimate and \eqref{intermediate estimate du/dt} to the definition of $J^{\mathrm{high}}$, we arrive at
\begin{align}
J^{\mathrm{high}} & \lesssim \sum_{k,\ell \in\mathbb{N}^n}  \int_{|\lambda| >\frac{1}{8\mu_k}} \big| \partial_t\widehat{u}(t,\lambda)_{ k,\ell}\big|^2 |\lambda|^n \,\mathrm{d}\lambda \notag \\
& \lesssim \mathrm{e}^{-\delta t} \sum_{k,\ell \in\mathbb{N}^n}  \int_{|\lambda| >\frac{1}{8\mu_k}}  \big(\mu_k |\lambda| \,| \widehat{u}_0(\lambda)_{ k,\ell}|^2+ | \widehat{u}_1(\lambda)_{ k,\ell}|^2\big) |\lambda|^n \,\mathrm{d}\lambda\notag\\
& \lesssim \mathrm{e}^{-\delta t} \sum_{k,\ell \in\mathbb{N}^n}  \int_{|\lambda| >\frac{1}{8\mu_k}}  \Big(\mu_k |\lambda| \, \big| \big(\widehat{u}_0(\lambda) e_k,e_\ell\big)_{L^2(\mathbb{R}^n)}\big|^2+\big| \big(\widehat{u}_1(\lambda) e_k,e_\ell\big)_{L^2(\mathbb{R}^n)}\big|^2\Big) |\lambda|^n \,\mathrm{d}\lambda   \notag  \\
& \lesssim \mathrm{e}^{-\delta t}   \int_{\mathbb{R}^*} \sum_{k,\ell \in\mathbb{N}^n} \Big(\mu_k |\lambda| \, \big| \big(\widehat{u}_0(\lambda) e_k,e_\ell\big)_{L^2(\mathbb{R}^n)}\big|^2+\big| \big(\widehat{u}_1(\lambda) e_k,e_\ell\big)_{L^2(\mathbb{R}^n)}\big|^2\Big) |\lambda|^n \,\mathrm{d}\lambda   \notag \\
&\lesssim \mathrm{e}^{-\delta t}   \int_{\mathbb{R}^*} \Big(  \| (\nabla_{\hor} u_0)^\wedge(\lambda)\|_{\HS[L^2(\mathbb{R}^n)]}^2+ \|  \widehat{u}_1(\lambda)\|_{\HS[L^2(\mathbb{R}^n)]}^2 \Big) |\lambda|^n \,\mathrm{d}\lambda \notag  \\
&   \approx   \mathrm{e}^{-\delta t} \Big(\| \nabla_{\hor} u_0\|_{L^2(\mathbf{H}_n)}^2+\|u_1\|_{L^2(\mathbf{H}_n)}^2\Big), \label{estimate J high}
\end{align} where we employed
\begin{align*}
\sum_{k,\ell \in\mathbb{N}^n}  \big| \sqrt{\mu_k |\lambda|} \, \big(\widehat{u}_0(\lambda) e_k,e_\ell\big)_{L^2(\mathbb{R}^n)}\big|^2 \approx \sum_{j=1}^n   \Big(  \| (X_j u_0)^\wedge(\lambda)\|_{\HS[L^2(\mathbb{R}^n)]}^2+ \|  (Y_j u_0)^\wedge(\lambda)\|_{\HS[L^2(\mathbb{R}^n)]}^2 \Big)
\end{align*} in the second last step  and Plancherel formula in the last step.
\begin{remark} \label{remark du/dt}
Also in this case, if we had used only $L^2$ regularity for any $\lambda\in\mathbb{R}^*$, we would have found the estimate
\begin{align*}
\| \partial_t u(t,\cdot)\|^2_{L^2(\mathbf{H}_n)}\lesssim  \Big(\| \nabla_{\hor} u_0\|_{L^2(\mathbf{H}_n)}^2+\|u_1\|_{L^2(\mathbf{H}_n)}^2\Big)
\end{align*} which excludes the possibility of $J^{\mathrm{low}}$ to be estimate by a singular term as $t\to 0^+ $.
\end{remark}
Due  Remark \ref{remark du/dt}, by \eqref{estimate J low} and \eqref{estimate J high} we get eventually \eqref{estimate partial t u}. 

\begin{remark}  Besides the uniform estimate in Remark \ref{remark du/dt}, which is necessary to exclude a singular behavior of $\| \partial_t u(t,\cdot)\|_{L^2(\mathbf{H}_n)}$ as $t\to 0^+$, by using just $L^2$ regularity in the estimate for $\partial_t u(t,\cdot)$ and making more sharp the intermediate steps, we may get some decay rate as well (of course, weaker than the one we got by assuming additional $L^1$ regularity). Indeed, keeping the estimate of $J^{\mathrm{low}}$ in \eqref{estimate for the remark} and applying \eqref{Plancherel formula} in spite of \eqref{Riemann Lebesgue inequality}, we have
\begin{align*}
J^{\mathrm{low}} & \lesssim \sum_{k \in\mathbb{N}^n} \mu_k^2 \int_{0<|\lambda|<\frac{1}{8\mu_k}} \mathrm{e}^{-2\mu_k|\lambda|  t}    \left( \|\widehat{u}_0(\lambda) e_k\|_{L^2(\mathbb{R}^n)}^2+ \|\widehat{u}_1(\lambda) e_k\|_{L^2(\mathbb{R}^n)}^2 \right) |\lambda|^{n+2} \,\mathrm{d}\lambda\\ 
& \qquad + \mathrm{e}^{-t} \sum_{k \in\mathbb{N}^n} \int_{0<|\lambda|<\frac{1}{8\mu_k}} \|\widehat{u}_1(\lambda) e_k\|_{L^2(\mathbb{R}^n)}^2 |\lambda|^n \,\mathrm{d}\lambda \\
& \lesssim  t^{-2} \sum_{k \in\mathbb{N}^n}  \int_{0<|\lambda|<\frac{1}{8\mu_k}}     \left( \|\widehat{u}_0(\lambda) e_k\|_{L^2(\mathbb{R}^n)}^2+ \|\widehat{u}_1(\lambda) e_k\|_{L^2(\mathbb{R}^n)}^2 \right) |\lambda|^{n} \,\mathrm{d}\lambda\\ 
& \qquad + \mathrm{e}^{-t} \sum_{k \in\mathbb{N}^n} \int_{0<|\lambda|<\frac{1}{8\mu_k}} \|\widehat{u}_1(\lambda) e_k\|_{L^2(\mathbb{R}^n)}^2 |\lambda|^n \,\mathrm{d}\lambda \\
& \lesssim   t^{-2}    \left( \|u_0\|_{L^2(\mathbf{H}_n)}^2+ \|u_1\|_{L^2(\mathbf{H}_n)}^2 \right) + \mathrm{e}^{-t}\,    \|u_1\|_{L^2(\mathbf{H}_n)}^2,
\end{align*} where in the second step we used the inequality $|\lambda|^2 \mathrm{e}^{-2\mu_k|\lambda|  t}  \lesssim (\mu_k t )^{-2} $. Combining the last estimate and \eqref{estimate J high} we get \eqref{estimate partial t u only L2}.
\end{remark}

\subsection{Estimate of the $L^2(\mathbf{H}_n)$ - norm of $\nabla_{\hor}u(t,\cdot)$}

In this section, we estimate the $L^2(\mathbf{H}_n)$ - norm of the horizontal gradient of $u(t,\cdot)$, that is,
\begin{align*}
\| \nabla_{\hor} u(t,\cdot)\|_{L^2(\mathbf{H}_n)}^2 \doteq \sum_{j=1}^n \Big( \| X_j u(t,\cdot)\|_{L^2(\mathbf{H}_n)}^2+ \| Y_j u(t,\cdot)\|_{L^2(\mathbf{H}_n)}^2\Big).
\end{align*} Let us fix $1\leq j\leq n$. We start with the estimate of the $L^2(\mathbf{H}_n)$ - norm of $X_j u(t,\cdot)$.
Combining \eqref{pi lambda Xj} and \eqref{e k derivative}, we get
\begin{align*}
\big((X_j u)^\wedge (t,\lambda)e_k,e_\ell\big)_{L^2(\mathbb{R}^n)} & = \sqrt{|\lambda|} \Big(\sqrt{\tfrac{k\cdot \epsilon_j}{2}} \,  \big(\widehat{u} (t,\lambda)e_{k-\epsilon_j},e_\ell\big)_{L^2(\mathbb{R}^n)} - \sqrt{\tfrac{1}{2}} \big(\widehat{u} (t,\lambda)e_{k+\epsilon_j},e_\ell\big)_{L^2(\mathbb{R}^n)}\Big) \\
& = \sqrt{\tfrac{|\lambda|}{2}} \big(\sqrt{k\cdot \epsilon_j}\,  \widehat{u} (t,\lambda)_{k-\epsilon_j,\ell} -  \widehat{u} (t,\lambda)_{k+\epsilon_j,\ell}\big)
\end{align*} for any $k,\ell\in \mathbb{N}^n$. Consequently, by Plancherel formula we have
\begin{align*}
\| X_j u(t,\cdot)\|_{L^2(\mathbf{H}_n)}^2 & =   c_n \int_{\mathbb{R}^*}\| (X_j u)^\wedge  (t,\lambda)\|^2_{\HS [L^2(\mathbb{R}^n)]} \, |\lambda|^n \,\mathrm{d}\lambda  \\
& = \frac{c_n}{2}  \int_{\mathbb{R}^*}  \sum_{k,\ell \in\mathbb{N}^n} \big|\big(\sqrt{k\cdot \epsilon_j}\,  \widehat{u} (t,\lambda)_{k-\epsilon_j,\ell} -  \widehat{u} (t,\lambda)_{k+\epsilon_j,\ell}\big)\big|^2 |\lambda|^{n+1} \,\mathrm{d}\lambda  \\
& =   \frac{c_n}{2} \! \sum_{k,\ell \in\mathbb{N}^n} \!\left( \int_{0<|\lambda|<\frac{1}{8\mu_{k-\epsilon_j}}}  + \int_{|\lambda| >\frac{1}{8\mu_{k-\epsilon_j}}}   \right) \! \big|\sqrt{k\cdot \epsilon_j}\,  \widehat{u} (t,\lambda)_{k-\epsilon_j,\ell} -  \widehat{u} (t,\lambda)_{k+\epsilon_j,\ell}\big|^2 |\lambda|^{n+1}\, \mathrm{d}\lambda  \\ & \doteq K^{\mathrm{low}}+ K^{\mathrm{high}}.
\end{align*} Note that we slightly changed the partition of the domain of integration with respect to the previous sections.
We estimate $K^{\mathrm{high}}$ first. Clearly,
\begin{align*}
K^{\mathrm{high}} & \lesssim  \sum_{k,\ell \in\mathbb{N}^n} \int_{|\lambda|>\frac{1}{8\mu_{k-\epsilon_j}}}  \Big( k\cdot \epsilon_j \, \big|\widehat{u} (t,\lambda)_{k-\epsilon_j,\ell}|^2+\big|  \widehat{u} (t,\lambda)_{k+\epsilon_j,\ell}\big|^2 \Big) |\lambda|^{n+1}\, \mathrm{d}\lambda.
\end{align*} From the representation formula \eqref{representation u hat components}, for $\lambda$ such that $|\lambda|>1/(8\mu_{k-\epsilon_j})>1/(8\mu_{k+\epsilon_j})$ we get the estimates
\begin{align*}
\big| \widehat{u}(t,\lambda)_{ k-\epsilon_j,\ell}\big| &\lesssim \mathrm{e}^{-\frac{\delta}{2} t} \big(| \widehat{u}_0(\lambda)_{ k-\epsilon_j,\ell}|+ (\mu_{k-\epsilon_j} |\lambda|)^{-1/2}\, | \widehat{u}_1(\lambda)_{ k-\epsilon_j,\ell}|\big), \\
\big| \widehat{u}(t,\lambda)_{ k+\epsilon_j,\ell}\big| &\lesssim \mathrm{e}^{-\frac{\delta}{2} t} \big( | \widehat{u}_0(\lambda)_{ k+\epsilon_j,\ell}|+ (\mu_{k+\epsilon_j} |\lambda|)^{-1/2}\,  | \widehat{u}_1(\lambda)_{ k+\epsilon_j,\ell}|\big),
\end{align*}  where $\delta>0$ is a suitable constant. Hence,
\begin{align*}
K^{\mathrm{high}} & \lesssim \mathrm{e}^{-\delta t}  \sum_{k,\ell \in\mathbb{N}^n} \int_{|\lambda|>\frac{1}{8\mu_{k-\epsilon_j}}}  k\cdot \epsilon_j \, \Big(|\lambda|\, | \widehat{u}_0(\lambda)_{ k-\epsilon_j,\ell}|^2+ (\mu_{k-\epsilon_j})^{-1}\, | \widehat{u}_1(\lambda)_{ k-\epsilon_j,\ell}|^2\Big) |\lambda|^{n}\, \mathrm{d}\lambda \\
& \qquad + \mathrm{e}^{-\delta t}  \sum_{k,\ell \in\mathbb{N}^n} \int_{|\lambda|>\frac{1}{8\mu_{k-\epsilon_j}}} \Big(|\lambda|\, | \widehat{u}_0(\lambda)_{ k+\epsilon_j,\ell}|^2+ (\mu_{k+\epsilon_j})^{-1}\, | \widehat{u}_1(\lambda)_{ k+\epsilon_j,\ell}|^2\Big) |\lambda|^{n} \, \mathrm{d}\lambda \\
 & \lesssim \mathrm{e}^{-\delta t}  \sum_{k,\ell \in\mathbb{N}^n} \int_{|\lambda|>\frac{1}{8\mu_{k-\epsilon_j}}} \Big( \big|( k\cdot \epsilon_j |\lambda|)^{1/2} \,  \widehat{u}_0(\lambda)_{ k-\epsilon_j,\ell}|^2+  | \widehat{u}_1(\lambda)_{ k-\epsilon_j,\ell}|^2\Big) |\lambda|^{n}\, \mathrm{d}\lambda \\
& \qquad + \mathrm{e}^{-\delta t}  \sum_{k,\ell \in\mathbb{N}^n} \int_{|\lambda|>\frac{1}{8\mu_{k-\epsilon_j}}} \Big( |\,|\lambda|^{1/2} \, \widehat{u}_0(\lambda)_{ k+\epsilon_j,\ell}|^2+  | \widehat{u}_1(\lambda)_{ k+\epsilon_j,\ell}|^2\Big) |\lambda|^{n} \, \mathrm{d}\lambda,
\end{align*} where we used the inequalities $(k\cdot \epsilon_j)/(\mu_{k-\epsilon_j})\lesssim 1$ and $\mu_{k+\epsilon_j} \gtrsim 1$ in the second step. Also,
\begin{align}
K^{\mathrm{high}} & \lesssim \mathrm{e}^{-\delta t} \int_{\mathbb{R}^*}  \sum_{k,\ell \in\mathbb{N}^n}  \Big( \big|( k\cdot \epsilon_j |\lambda|)^{1/2} \,  \widehat{u}_0(\lambda)_{ k-\epsilon_j,\ell}|^2+  |\,|\lambda|^{1/2} \, \widehat{u}_0(\lambda)_{ k+\epsilon_j,\ell}|^2  \Big) |\lambda|^{n}\, \mathrm{d}\lambda \notag \\
& \qquad + \mathrm{e}^{-\delta t}  \int_{\mathbb{R}^*} \sum_{k,\ell \in\mathbb{N}^n}   \Big( | \widehat{u}_1(\lambda)_{ k-\epsilon_j,\ell}|^2 +  | \widehat{u}_1(\lambda)_{ k+\epsilon_j,\ell}|^2\Big) |\lambda|^{n} \, \mathrm{d}\lambda \notag \\
& \lesssim \mathrm{e}^{-\delta t} \Big( \| (X_j+iY_j)u_0\|^2_{L^2(\mathbf{H}_n)}+\| (X_j-iY_j)u_0\|^2_{L^2(\mathbf{H}_n)}+\| u_1\|^2_{L^2(\mathbf{H}_n)}\Big) \notag \\
& \lesssim \mathrm{e}^{-\delta t} \Big( \| X_j u_0\|^2_{L^2(\mathbf{H}_n)}+\| Y_j u_0\|^2_{L^2(\mathbf{H}_n)}+\| u_1\|^2_{L^2(\mathbf{H}_n)}\Big). \label{estimate K high}
\end{align} Note that in the previous chain of inequalities, according to \eqref{pi lambda Xj} and \eqref{pi lambda Yj}, we used 
\begin{align*}
\sum_{k,\ell \in\mathbb{N}^n} & \Big( \big|( k\cdot \epsilon_j |\lambda|)^{1/2} \,  \widehat{u}_0(\lambda)_{ k-\epsilon_j,\ell}|^2+  |\,|\lambda|^{1/2} \, \widehat{u}_0(\lambda)_{ k+\epsilon_j,\ell}|^2  \Big) \\
& = 2 \sum_{k,\ell \in\mathbb{N}^n}  \Big( \big| \big(\big(( X_j + i Y_j)u_0\big)^\wedge(\lambda) e_k,e_\ell \big)_{L^2(\mathbb{R}^n)}\big|^2+  \big|\big(\big(( X_j - i Y_j)u_0\big)^\wedge(\lambda) e_k,e_\ell \big)_{L^2(\mathbb{R}^n)}\big|^2  \Big) \\
& = 2 \, \Big(\big\| \big((X_j+iY_j)u_0\big)^\wedge(\lambda) \big\|^2_{\HS[L^2(\mathbb{R}^n)]}+\big\| \big((X_j-iY_j)u_0 \big)^\wedge (\lambda)\big\|^2_{\HS[L^2(\mathbb{R}^n)]}\Big).
\end{align*}
Next, we estimate $K^{\mathrm{low}}$. Combining
\begin{align*}
K^{\mathrm{low}} & \lesssim  \sum_{k,\ell \in\mathbb{N}^n} \int_{0<|\lambda|<\frac{1}{8\mu_{k-\epsilon_j}}}  \Big( k\cdot \epsilon_j \, \big|\widehat{u} (t,\lambda)_{k-\epsilon_j,\ell}|^2+\big|  \widehat{u} (t,\lambda)_{k+\epsilon_j,\ell}\big|^2 \Big) |\lambda|^{n+1}\, \mathrm{d}\lambda
\end{align*} and 
\begin{align*}
 \big|\widehat{u} (t,\lambda)_{k-\epsilon_j,\ell}|^2 & \lesssim \mathrm{e}^{- 2 \mu_{k-\epsilon_j}|\lambda| t} \left(|\widehat{u}_0(\lambda)_{k-\epsilon_j,\ell}|^2+|\widehat{u}_1(\lambda)_{k-\epsilon_j,\ell}|^2\right)   \\
 \big|\widehat{u} (t,\lambda)_{k+\epsilon_j,\ell}|^2 &  \lesssim \mathrm{e}^{- 2 \mu_{k+\epsilon_j}|\lambda| t} \left(|\widehat{u}_0(\lambda)_{k+\epsilon_j,\ell}|^2+|\widehat{u}_1(\lambda)_{k+\epsilon_j,\ell}|^2\right)  \\ & \lesssim \mathrm{e}^{- 2 \mu_{k-\epsilon_j}|\lambda| t} \left(|\widehat{u}_0(\lambda)_{k+\epsilon_j,\ell}|^2+|\widehat{u}_1(\lambda)_{k+\epsilon_j,\ell}|^2\right) 
\end{align*} for $|\lambda|<1/(8\mu_{k-\epsilon_j})$, where we used again \eqref{sqrt(1-4z) estimate} and $\mu_{k+\epsilon_j}>\mu_{k-\epsilon_j}$, we obtain
\begin{align}
K^{\mathrm{low}} & \lesssim  \sum_{k,\ell \in\mathbb{N}^n}  k\cdot \epsilon_j  \int_{0<|\lambda|<\frac{1}{8\mu_{k-\epsilon_j}}} \mathrm{e}^{- 2 \mu_{k-\epsilon_j}|\lambda| t}  \Big(|\widehat{u}_0(\lambda)_{k-\epsilon_j,\ell}|^2+|\widehat{u}_1(\lambda)_{k-\epsilon_j,\ell}|^2 \Big) |\lambda|^{n+1}\, \mathrm{d}\lambda  \notag\\
& \qquad + \sum_{k,\ell \in\mathbb{N}^n}  \int_{0<|\lambda|<\frac{1}{8\mu_{k-\epsilon_j}}} \mathrm{e}^{- 2 \mu_{k-\epsilon_j}|\lambda| t}  \Big( |\widehat{u}_0(\lambda)_{k+\epsilon_j,\ell}|^2+|\widehat{u}_1(\lambda)_{k+\epsilon_j,\ell}|^2\Big) |\lambda|^{n+1}\, \mathrm{d}\lambda \notag \\
 & \lesssim  \sum_{k\in\mathbb{N}^n}  k\cdot \epsilon_j  \int_{0<|\lambda|<\frac{1}{8\mu_{k-\epsilon_j}}} \mathrm{e}^{- 2 \mu_{k-\epsilon_j}|\lambda| t}  \Big( \|\widehat{u}_0(\lambda) e_{k-\epsilon_j}\|^2_{L^2(\mathbb{R}^n)}+\|\widehat{u}_1(\lambda) e_{k-\epsilon_j}\|^2_{L^2(\mathbb{R}^n)}\Big) |\lambda|^{n+1}\, \mathrm{d}\lambda \notag  \\
& \qquad + \sum_{k,\in\mathbb{N}^n}  \int_{0<|\lambda|<\frac{1}{8\mu_{k-\epsilon_j}}} \mathrm{e}^{- 2 \mu_{k-\epsilon_j}|\lambda| t}  \Big(\|\widehat{u}_0(\lambda) e_{k+\epsilon_j}\|^2_{L^2(\mathbb{R}^n)}+\|\widehat{u}_1(\lambda) e_{k+\epsilon_j}\|^2_{L^2(\mathbb{R}^n)}\Big) |\lambda|^{n+1}\, \mathrm{d}\lambda , \label{estimate for the remark 2}
\end{align} where in the second inequality we used Parseval's identity. Therefore, by \eqref{estimate u hat e k} it results
\begin{align}
K^{\mathrm{low}} & \lesssim  \sum_{k\in\mathbb{N}^n}  (k\cdot \epsilon_j +1) \int_{0}^{\frac{1}{8\mu_{k-\epsilon_j}}} \mathrm{e}^{- 2 \mu_{k-\epsilon_j} \lambda t}  \lambda^{n+1}\, \mathrm{d}\lambda  \, \Big( \|u_0\|^2_{L^1(\mathbf{H}_n)}+\|u_1\|^2_{L^1(\mathbf{H}_n)}\Big) \notag \\
& \lesssim  t^{-(n+2)}\sum_{k\in\mathbb{N}^n}  (k\cdot \epsilon_j +1) \mu_{k-\epsilon_j}^{-(n+2)}\int_{0}^{\frac{t}{4}} \mathrm{e}^{- \theta}  \theta^{n+1}\, \mathrm{d}\theta  \, \Big( \|u_0\|^2_{L^1(\mathbf{H}_n)}+\|u_1\|^2_{L^1(\mathbf{H}_n)}\Big)\notag  \\
& \lesssim  t^{-\frac{\mathcal{Q}}{2}-1}\sum_{k\in\mathbb{N}^n}  (k\cdot \epsilon_j +1) (|k|-1+n)^{-(n+2)} \, \Big( \|u_0\|^2_{L^1(\mathbf{H}_n)}+\|u_1\|^2_{L^1(\mathbf{H}_n)}\Big) \notag \\
& \lesssim  t^{-\frac{\mathcal{Q}}{2}-1}\, \Big( \|u_0\|^2_{L^1(\mathbf{H}_n)}+\|u_1\|^2_{L^1(\mathbf{H}_n)}\Big), \label{estimate K low}
\end{align} where we performed the change of variables $\theta=2\mu_{k-\epsilon_j} \lambda t$ in the second line and we employed the convergence of the series
\begin{align*}
\sum_{k\in\mathbb{N}^n}  (k\cdot \epsilon_j +1) (|k|-1+n)^{-(n+2)}<\infty
\end{align*}
in the last one. Similarly as in Remarks \ref{remark u} and \ref{remark du/dt}, we can exclude a singular coefficients with respect to $t$ as $t\to 0^+$ for $K^{\mathrm{low}}$. 
Summarizing, we proved
\begin{align*}
\| X_j u(t,\cdot)\|_{L^2(\mathbf{H}_n)}^2 \lesssim  (1+t)^{-\frac{\mathcal{Q}}{2}-1}  \Big(  \! \|u_0\|^2_{L^1(\mathbf{H}_n)}+\| X_j u_0\|^2_{L^2(\mathbf{H}_n)}+\| Y_j u_0\|^2_{L^2(\mathbf{H}_n)}+\|u_1\|^2_{L^1(\mathbf{H}_n)}+\| u_1\|^2_{L^2(\mathbf{H}_n)}\! \Big).
\end{align*} In a completely analogous way it is possible to prove the same kind of estimate for $\| Y_j u(t,\cdot)\|_{L^2(\mathbf{H}_n)}^2 $. So, this complete the proof of \eqref{estimate nabla hor u}.

\begin{remark} If we continue the estimate \eqref{estimate for the remark 2} for $K^{\mathrm{low}}$ by using \eqref{Plancherel formula} in place of \eqref{Riemann Lebesgue inequality}, we arrive at
\begin{align*}
K^{\mathrm{low}} & \lesssim  t^{-1} \sum_{k\in\mathbb{N}^n}  (k\cdot \epsilon_j ) \mu_{k-\epsilon_j}^{-1} \int_{0<|\lambda|<\frac{1}{8\mu_{k-\epsilon_j}}}  \Big( \|\widehat{u}_0(\lambda) e_{k-\epsilon_j}\|^2_{L^2(\mathbb{R}^n)}+\|\widehat{u}_1(\lambda) e_{k-\epsilon_j}\|^2_{L^2(\mathbb{R}^n)}\Big) |\lambda|^{n}\, \mathrm{d}\lambda   \\
& \qquad + t^{-1} \sum_{k,\in\mathbb{N}^n} \mu_{k-\epsilon_j}^{-1}  \int_{0<|\lambda|<\frac{1}{8\mu_{k-\epsilon_j}}} \Big(\|\widehat{u}_0(\lambda) e_{k+\epsilon_j}\|^2_{L^2(\mathbb{R}^n)}+\|\widehat{u}_1(\lambda) e_{k+\epsilon_j}\|^2_{L^2(\mathbb{R}^n)}\Big) |\lambda|^{n}\, \mathrm{d}\lambda  \\
 & \lesssim  t^{-1}    \int_{\mathbb{R}^*}  \sum_{k\in\mathbb{N}^n} \Big( \|\widehat{u}_0(\lambda) e_{k}\|^2_{L^2(\mathbb{R}^n)}+\|\widehat{u}_1(\lambda) e_{k}\|^2_{L^2(\mathbb{R}^n)}\Big) |\lambda|^{n}\, \mathrm{d}\lambda  \\
 & \lesssim  t^{-1}  \Big( \|u_0\|^2_{L^2(\mathbf{H}_n)}+\|u_1\|^2_{L^2(\mathbf{H}_n)}\Big)
\end{align*} where in the second step we applied the inequality
$$(k\cdot \epsilon_j ) \mu_{k-\epsilon_j}^{-1} \leq \frac{|k|}{2(|k|-1)+n}\lesssim 1$$ and the fact that the set of the eigenvalues of the harmonic oscillator is bounded from below by a positive constant. Thus, combining the previous estimate for $K^{\mathrm{low}}$ and \eqref{estimate K high}, we obtain \eqref{estimate nabla hor u only L2}.
\end{remark}

\section{Final remarks and future applications}

In Theorem \ref{Main Thm}, we restricted our considerations to the $L^2(\mathbf{H}_n)$ - estimate for the horizontal gradient of $u(t,\cdot)$. This choice seems to be quite natural due to the differential operator on the left-hand side of \eqref{damped wave equation Heisenberg}. Nevertheless, one might wonder what happens if we apply the machinery developed in Section \ref{Section Proof Main Thm} to the derivative $T u(t,\cdot)$ (as in the introduction, $T$ denotes here the generator of the second layer of the Lie algebra $\mathfrak{h}_n$). The infinitesimal generator of $\pi_\lambda$ acts on $T$ as follows:
\begin{align*}
\mathrm{d}\pi_\lambda(T)= i\lambda \, \mathrm{Id}.
\end{align*} As in Section \ref{Section Proof Main Thm}, if we employ Plancherel formula, we split the integral of containing the Hilbert-Schmidt norm of $(T u)^\wedge(t,\lambda)=i\lambda \widehat{u}(t,\lambda)$ in two regions and we use $L^1(\mathbf{H}_n)$ regularity for the data in the low $\lambda$ region and $L^2(\mathbf{H}_n)$ regularity for the data in the high $\lambda$ region, then, we end up with the estimate
\begin{align}\label{estimate Tu}
\| T u(t,\cdot)\|_{L^2(\mathbf{H}_n)} \lesssim (1+t)^{-\frac{\mathcal{Q}}{4}-1}  \Big(\! \|u_0\|_{L^1(\mathbf{H}_n)}+\| T u_0\|_{L^2(\mathbf{H}_n)}+\|u_1\|_{L^1(\mathbf{H}_n)}+\| T^{1/2}u_1\|_{L^2(\mathbf{H}_n)}+\| u_1\|_{L^2(\mathbf{H}_n)}\!\Big),
\end{align} where 
\begin{align*}
\| T^{1/2}u_1\|_{L^2(\mathbf{H}_n)}^2  \doteq c_n \int_{\mathbb{R}^*} \| \widehat{u}_1(\lambda)\|_{\HS [L^2(\mathbb{R}^n)]}^2 \, |\lambda|^{n+1} \, \mathrm{d}\lambda.
\end{align*}
So, in \eqref{estimate Tu} we get the same decay rate as in \eqref{estimate partial t u}, which is faster than the one for the horizontal gradient of $u(t,\cdot)$ in \eqref{estimate nabla hor u}, but we are forced somehow to require more regularity for $u_1$ on $L^2$ level. 

Concerning future applications for the tools derived in this work, in the forthcoming paper \cite{GP19DW}, the decay estimates proved in Theorem \ref{Main Thm} are going to be used for proving a global existence result for small data solutions for the semilinear damped wave equation with power nonlinearity on the Heisenberg group, namely, for the Cauchy problem
\begin{align}\label{damped wave equation Heisenberg semilinear}
\begin{cases}
\partial_t^2 u(t,\eta)-\Delta_{\hor} u(t,\eta)+\partial_t u(t,\eta)= |u|^p, & \eta\in \mathbf{H}_n, \ t>0, \\
u(0,\eta)= u_0(\eta), & \eta\in \mathbf{H}_n, \\
\partial_t u(0,\eta)= u_1(\eta), & \eta\in \mathbf{H}_n,
\end{cases}
\end{align}  requiring a suitable lower bound for $p>1$.
By following the main ideas from \cite{TY01,IT05} in the Euclidean case, an exponential weight function will be introduced. Hence, the global existence of small data solutions for exponents $p$ in the super-Fujita case will be shown in the corresponding weighted energy space. Note that this will provide a critical exponent of Fujita-type as for the corresponding semilinear heat equation on the Heisenberg group or on more general nilpotent Lie groups (cf. \cite{Ruz18,GP19,GP19Car}).

\section*{Acknowledgments}

%V. Georgiev is supported in part by
%GNAMPA - Gruppo Nazionale per l'Analisi Matematica,
%la Probabilit\`a e le loro Applicazioni,
%by Institute of Mathematics and Informatics,
%Bulgarian Academy of Sciences and Top Global University Project, Waseda University,  by the University of Pisa, Project PRA 2018 49.
The author is supported by the University of Pisa, Project PRA 2018 49. The author thanks Vladimir Georgiev (University of Pisa) for fruitful discussions on the topics of this work.



%%       \citet{<label>} ==> Jones et al. [21]
%%       \citep{<label>} ==> [21]
%%

  \bibliographystyle{elsarticle-num-names} 


%\section*{References}

\end{document}